\documentclass[reqno]{amsart}
\usepackage{amsmath,amssymb,amsfonts}
\usepackage{euscript}
\usepackage{enumerate}
\usepackage{pstricks}
\usepackage{verbatim}
\newtheorem{theorem}{Theorem}
\newtheorem{lemma}{Lemma}

\newtheorem{example}{Example}

\renewcommand{\phi}{\varphi}
\newcommand{\eps}{\varepsilon}
\newcommand{\lbd}{\lambda}
\DeclareMathOperator{\Id}{Id}
\DeclareMathOperator{\e}{e}
\newcommand{\prts}[1]{\left(#1\right)}
\newcommand{\prtsr}[1]{\left[#1\right]}
\newcommand{\abs}[1]{\left|#1\right|}
\newcommand{\set}[1]{\left\{#1\right\}}
\newcommand{\setm}[1]{\setminus\set{#1}}
\newcommand{\pfrac}[2]{\prts{\dfrac{#1}{#2}}}
\newcommand{\dsum}{\displaystyle\sum}
\newcommand{\dt}{\, dt}
\usepackage{dsfont}
   \newcommand{\N}{\ensuremath{\mathds N}}
   \newcommand{\R}{\ensuremath{\mathds R}}
\def\cA{\EuScript{A}}
\def\cB{\EuScript{B}}
\def\cF{\EuScript{F}}
\def\cV{\EuScript{V}}
\def\cX{\EuScript{X}}
\def\cP{\EuScript{P}}

\begin{document}
\title [Nonuniform $\mathbf{(\mu,\nu)}$-dichotomies and local dynamics ...]
   {Nonuniform $\mathbf{(\mu,\nu)}$-dichotomies and local dynamics of
   difference equations}
\author[Ant\'onio J. G. Bento]{Ant\'onio J. G. Bento}
\address{Ant\'onio J. G. Bento\\
   Departamento de Matem\'atica\\
   Universidade da Beira Interior\\
   6201-001 Covilh\~a\\
   Portugal}
\email{bento@ubi.pt}
\author{C\'esar M. Silva}
\address{C\'esar M. Silva\\
   Departamento de Matem\'atica\\
   Universidade da Beira Interior\\
   6201-001 Covilh\~a\\
   Portugal}
\email{csilva@ubi.pt}
\urladdr{www.mat.ubi.pt/~csilva}
\date{\today}
\begin{abstract}
   We obtain a local stable manifold theorem for perturbations of nonautonomous
   linear difference equations possessing a very general type of nonuniform
   dichotomy, possibly with different growth rates in the uniform and
   nonuniform parts. We note that we consider situations were the classical
   Lyapunov exponents can be zero. Additionally, we study how the manifolds
   decay along the orbit of a point as well as the behavior under perturbations
   and give examples of nonautonomous linear difference equations that admit
   the dichotomies considered.
\end{abstract}
\maketitle
\section{Introduction}
   The main purpose of this paper is to discuss, in a Banach space $X$, the
   existence of stable manifolds for a general family of perturbations of
   nonautonomous linear difference equation
     $$x_{m+1}=A_m x_m + f_m(x_m), \quad m \in \N ,$$
   assuming that the perturbations $f_m \colon X \to X$ verify $f_m(0)=0$,
      $$\|f_m(u)- f_m(v)\| \le c \| u-v\|(\|u\| + \|v\|)^q, \quad m \in \N,$$
   for some constants $c > 0$ and $q>1$ and for each $u,v \in X$, and that the
   linear equation
     $$x_{m+1}=A_m x_m, \quad m \in \N,$$
   admits a very general type of nonuniform dichotomy given by arbitrary rates
   of growth.

The notion of uniform exponential dichotomy was introduced by Perron
in~\cite{Perron-MZ-1930} and constitutes a very important tool in the theory of
difference and differential equations, particularly in the study of invariant
manifolds. In spite of being used in a wide range of situations, sometimes this
notion is too demanding and it is of interest to consider more general kinds of
hyperbolic behavior. A much more general type of dichotomy, allowing the rates
of growth to vary along the trajectory of a point, is the notion of nonuniform
exponential dichotomy that was introduced by Barreira and Valls in the context
of nonautonomous differential equations in~\cite{Barreira-Valls-JDE-2006} and
that was inspired both in Perron's classical notion of exponential dichotomy
and in the notion of nonuniformly hyperbolic trajectory introduced by Pesin
in~\cite{Pesin-IANSSSR-1976,Pesin-UMN-1977,Pesin-IANSSSR-1977}. In the context
of difference equations, it was also introduced a notion of nonuniform
exponential dichotomy in~\cite{Barreira-Valls-DCDS-2006}.

The study of stable manifolds in the nonuniform context has a long history,
starting with a famous theorem on existence of stable manifolds for
nonuniformly hyperbolic trajectories, in the finite dimensional setting, proved
by Pesin~\cite{Pesin-IANSSSR-1976}. In~\cite{Ruelle-IHESPM-1979} Ruelle gave an
alternative proof of this theorem based on the study of perturbations of
products of matrices occurring in Oseledets' multiplicative ergodic
theorem~\cite{Oseledets-TMMS-1968} and, inspired in the classical work of
Hadamard, Pugh and Shub~\cite{Pugh-Shub-TAMS-1989} proved the same result using
graph transform techniques. In Hilbert spaces and under some compactness
assumptions, Ruelle~\cite{Ruelle-AM-1982} obtained a version of the stable
manifold theorem, following his approach in~\cite{Ruelle-IHESPM-1979}. Versions
of this theorem for transformations in Banach spaces, were established first by
Ma\~n\'e~\cite{Mane-LNM-1983} under some compactness and invertibility
assumptions and then by Thieullen~\cite{Thieullen-AIHPAN-1987} under weaker
hypothesis.

Stable manifold were also obtained for perturbations of nonautonomous linear
differential equations and for perturbations of nonautonomous linear difference
equations, assuming respectively that the linear differential equation and
linear difference equation admit a nonuniform exponential dichotomy. We refer
the reader to the book~\cite{Barreira-Valls-LNM-2008}, where the obtention of
stable manifolds for perturbations of linear differential equations admitting
the existence of nonuniform exponential dichotomies is discussed, and also
to~\cite{Barreira-Valls-DCDS-2006, Barreira-Silva-Valls-JDDE-2008,
Barreira-Silva-Valls-JLMS-2008} for a related discussion in the context of
difference equations.

Recently, invariant stable manifolds were obtained for perturbations of
nonautonomous linear difference and differential equations, assuming the
existence of nonuniform dichotomies that are not exponential. In particular, in
the discrete time setting, assuming the existence of a some type of polynomial
dichotomy for a nonautonomous linear difference equation, it was established
in~\cite{Bento-Silva-JFA-2009} the existence of local stable manifolds for a
certain class of perturbations and, for a more restricted class, there were
also obtained global stable manifolds.

Our result can be seen as a discrete counterpart of the results obtained
in~\cite{Bento-Silva-arXiv:1007.5039v1} for nonautonomous differential
equations and we emphasize that the stable manifold theorem for perturbations
of linear difference equations with nonuniform exponential dichotomies
in~\cite{Barreira-Valls-DCDS-2006} is included in our theorem as a very
particular case and our result also includes as particular cases stable
manifold theorems for polynomial dichotomies, as well as many other situations
where the classical Lyapunov exponent is zero. We stress that, to the best of
our knowledge, in the context of perturbations of nonautonomous linear
difference equations that admit a non-exponential nonuniform dichotomy, our
result is the first one addressing the existence of local stable manifolds for
the general class of perturbations above. In particular, it is new even for
nonuniform polynomial dichotomies (in~\cite{Bento-Silva-JFA-2009} it was
already considered the polynomial case but the type of nonuniform dichotomies
considered there were different from the ones considered here). In the context
of differential equations and under the existence of nonuniform polynomial
dichotomies it were also obtained local and global stable manifolds in
\cite{Bento-Silva-QJM}.

As mentioned, the type of dichotomies considered in this paper are very
general, allowing different rates of growth for the uniform and the nonuniform
parts and thus, to establish the existence of stable manifolds, we must assume
conditions relating the rate of decay of the some balls in the stable spaces
and the growth rates.

To highlight the generality of this concept of dichotomy, we discuss some
families of new examples that verify the hypothesis in our main result.
Additionally, we obtain a upper bound for the decay of solutions along the
stable manifolds and we study how the stable manifolds vary with the
perturbations by giving bounds, in some suitable metric, on the distances
between the functions whose graphs are the stable manifolds.

The content of the paper is the following: in Section~\ref{Section:Main-Result}
we introduce some notation, the main definitions and state the main theorem;
next, in Section~\ref{Section:Examples} we present some examples; then, in
Section~\ref{Section:Proof}, we prove the main theorem; finally, in
Section~\ref{section:perturbations} we study how the manifolds obtained vary
with the perturbations considered.
\section{Main Result} \label{Section:Main-Result}
We say that an increasing sequence $\mu=\prts{\mu_n}_{n \in \N_0}$ is a
\textit{growth rate} if $\mu_0 \ge 1$ and
   $\lim\limits_{n \to + \infty} \mu_n = +\infty.$

Let $\mu = \prts{\mu_n}_{n \in \N_0}$ and $\nu = \prts{\nu_n}_{n \in \N_0}$ be
growth rates and let $B(X)$ be the space of bounded linear operators in a
Banach space~$X$. Given a sequence $(A_n)_{n\in\N}$ of invertible operators of
$B(X)$ and putting
   $$ \cA_{m,n} =
      \begin{cases}
         A_{m-1} \cdots A_n & \text{ if } m>n, \\
         \text{Id} & \text{ if } m=n,
      \end{cases}$$
we say that the linear difference equation
\begin{equation} \label{eq:lin:dif}
   x_{m+1} = A_m x_m, \ m \in \N
\end{equation}
admits a \textit{nonuniform $(\mu,\nu)$-dichotomy} if there exist projections
$P_m$, $m\in\N$, such that
   $$ P_m \cA_{m,n} = \cA_{m,n} P_n,\quad m, n\in\N,$$
and constants $a < 0\le b$, $\eps \ge 0$ and $D\ge 1$ such that for every $n
\in \N$ and every $m \ge n$,
\begin{align}
   & \|\cA_{m,n} P_n\|
      \le D \pfrac{\mu_m}{\mu_{n-1}}^a \nu_{n-1}^\eps, \label{eq:dich-1}\\
   & \|\cA_{m,n}^{-1}Q_m\|
      \le D \pfrac{\mu_{m-1}}{\mu_n}^{-b} \nu_{m-1}^\eps, \label{eq:dich-2}
\end{align}
where $Q_m=\Id-P_m$ is the complementary projection. When $\eps = 0$ we say
that we have a \textit{uniform $\mu$-dichotomy} or simply a
\textit{$\mu$-dichotomy}.

In these conditions we define, for each $n \in \N$, the linear subspaces
$E_n=P_n (X)$ and $F_n=Q_n(X)$. As usual, we identify the vector spaces $E_n
\times F_n$ and $E_n \oplus F_n$ as the same vector space.

We are going to address the problem of existence of stable manifolds of the
difference equation
   $$ x_{m+1} = A_m x_m + f_m(x_m), \ m \in \N,$$
where $f_m : X \to X$ are perturbations for which there are constants $c>0$ and
$q>1$ such that
\begin{align}
   & f_m(0)=0, \label{cond-f-0}\\
   & \|f_m(u)- f_m(v)\| \le c \| u-v\|(\|u\| + \|v\|)^q \label{cond-f-1}
\end{align}
for every $m \in \N$ and every $u, v \in X$. Note that making $v = 0$ in
\eqref{cond-f-1} we have
\begin{equation} \label{cond-f-1a}
   \|f_m(u)\| \le c \|u\|^{q+1}
\end{equation}
for every $m \in \N$ and every $u \in X$.

Given $n\in \N$ and $v_n=(\xi,\eta)\in E_n \times F_n$, for each $m>n$ we
write
\begin{equation}\label{eq:traj}
   v_m
   =\cF_{m,n}(v_n)
   = \cF_{m,n}(\xi,\eta)
   =(x_m,y_m) \in E_m \times F_m,
\end{equation}
with
\begin{equation} \label{eq:dyn}
   \cF_{m,n} =
      \begin{cases}
         (A_{m-1}+f_{m-1})\circ \cdots\circ (A_n+f_n) & \text{ if } m>n, \\
         \text{Id} & \text{ if } m=n.
      \end{cases}
\end{equation}

We denote by $B_n(r)$ the open ball of $E_n$ centered at zero and with radius
$r>0$. Fix now $\delta>0$ and let $\beta=\prts{\beta_n}_{n \in \N}$ be a
positive sequence. We denote by $\cX_{\delta,\beta}$ the space of sequences
$(\phi_n)_{n \in \N}$ of continuous functions $\phi_n\colon B_n(\delta
\beta_n)\to F_n$ such that
\begin{align}
   & \phi_n(0)=0 \label{cond-phi-0} \\
   & \| \phi_n(\xi) - \phi_n(\bar\xi) \| \le \| \xi - \bar\xi\|
   \label{cond-phi-1}
\end{align}
for every $\xi$, $\bar{\xi} \in B_n(\delta \beta_n)$ and every $n \in \N$.
Given $(\phi_n)_{n \in \N} \in \cX_{\delta,\beta}$, for each $n \in \N$, we
consider the graph
\begin{equation}\label{def:V_phi,n,delta,beta}
   \cV_{\phi,n,\delta,\beta}
   = \set{(\xi,\phi_n(\xi)): \xi \in B_n(\delta \beta_n)},
\end{equation}
that we call \textit{local stable manifold}.

We now state the result on the existence of local stable manifolds and its
proof will be given in Section~\ref{Section:Proof}.

\begin{theorem} \label{thm:local}
   Given a Banach space $X$, let $f_m : X \to X$ be a sequence of functions
   satisfying \eqref{cond-f-0} and \eqref{cond-f-1} for some $c > 0$ and $q >
   1$. Suppose equation~\eqref{eq:lin:dif} admits a nonuniform
   $(\mu,\nu)$-dichotomy for some growth rates $\mu$ and $\nu$, $D \ge 1$, $a <
   0 \le b$ and $\eps \ge 0$. Assume that
   \begin{equation} \label{eq:CondicaoTeo}
      \lim_{m \to +\infty} \mu_m^a \mu_{m-1}^{-b} \nu_{m-1}^\eps = 0
   \end{equation}
   and that
   \begin{equation} \label{int-conv}
      \dsum_{k=1}^{+\infty} \mu_k^{aq} \nu_k^\eps \text{ is convergent}.
   \end{equation}
   Define the sequences $\beta=\prts{\beta_m}_{m \in \N}$ and
   $\tilde\beta=\prts{\tilde\beta_m}_{m \in \N}$ by
   \begin{equation}\label{def:beta}
      \beta_m
      = \dfrac{\mu_{m-1}^a}{\nu_{m-1}^{\eps(1+1/q)}
         \prts{\dsum_{k=m}^{+\infty} \mu_k^{aq} \nu_k^\eps}^{1/q}}
      \ \ \ \text{ and }  \ \ \
      \tilde{\beta}_m = \beta_m \nu_{m-1}^{-\eps}
   \end{equation}
   and suppose that there is a constant $K \ge 1$ such that
   \begin{equation}\label{eq:decreasing}
      \dfrac{\mu_m^a \beta_m^{-1}}{\mu_{n-1}^a \beta_n^{-1}} \le K
      \text{ for every } n \in \N \text{ and every } m \ge n.
   \end{equation}
   Then, for every $C > D$, choosing $\delta > 0$ sufficiently small, there is
   a unique $\phi \in \cX_{\delta,\beta}$ such that
   \begin{equation} \label{thm:local:invar}
      \cF_{m,n}(\cV_{\phi,n,\delta/(CK),\tilde{\beta}})
      \subseteq \cV_{\phi,m,\delta,\beta} \
      \text{for every } n \in \N \text{ and every } m \ge n.
   \end{equation}
   where $\cV_{\phi,n,\delta/(CK),\tilde{\beta}}$ and
   $\cV_{\phi,m,\delta,\beta}$ are given by~\eqref{def:V_phi,n,delta,beta}.
   Furthermore, given $n \in \N$, we have
   \begin{equation}\label{thm:ine:norm:F_mn(xi...)-F_mn(barxi...)}
      \|\cF_{m,n}(\xi,\phi_n(\xi)) -\cF_{m,n}(\bar \xi,\phi_n(\bar \xi)) \|
      \le 2 C \pfrac{\mu_m}{\mu_{n-1}}^a \nu_{n-1}^\eps \, \|\xi-\bar \xi\|.
   \end{equation}
   for every $m \ge n$ and $\xi, \bar\xi \in
   B_n(\delta\tilde\beta_n/(CK))$.
\end{theorem}
\section{Examples} \label{Section:Examples}
In this section we will illustrate our main result with some examples. Firstly,
we will give examples of linear difference equations admitting nonuniform
$(\mu, \nu)$-dichotomies for any growth rates $\mu$ and $\nu$. Secondly, we
show that the nonuniform exponential result obtained
in~\cite{Barreira-Valls-DCDS-2006} is a particular case of our theorem and
finally we highlight two new settings were our main theorem can be applied.

\begin{example}
   Given $a < 0 \le b$ and $\eps \ge 0$, let $\prts{A_n}_{n \in \N}$ be the
   sequence of bounded linear operators $A_n : \R^2 \to \R^2$ given by the
   diagonal matrices
      $$ A_n = \prtsr{
         \begin{array}{cc}
            \pfrac{\mu_{n+1}}{\mu_{n-1}}^a \, \
               \pfrac{\nu_n^{\cos(n\pi)-1}}
               {\nu_{n-1}^{\cos((n-1)\pi)-1}}^{\eps/2}
            & 0 \\
            0 &
               \pfrac{\mu_{n+1}}{\mu_n}^b \, \
               \pfrac{\nu_n^{\cos(n\pi)-1}}
               {\nu_{n-1}^{\cos((n-1)\pi)-1}}^{\eps/2}
         \end{array}}$$
   where $\displaystyle\mu = \prts{\mu_n}_{n \in \N_0}$ and $\displaystyle\nu =
   \prts{\nu_n}_{n \in \N_0}$ are two growth rates. Then
      $$ \cA_{m,n} = \prtsr{
         \begin{array}{cc}
            \pfrac{\mu_{m-1} \mu_m}{\mu_{n-1}\mu_n}^a
            \pfrac{\nu_{m-1}^{\cos((m-1)\pi)-1}}
               {\nu_{n-1}^{\cos((n-1)\pi)-1}}^{\eps/2}
            & 0 \\
            0 &
               \pfrac{\mu_m}{\mu_n}^b \, \
               \pfrac{\nu_{m-1}^{\cos((m-1)\pi)-1}}
               {\nu_{n-1}^{\cos((n-1)\pi)-1}}^{\eps/2}
         \end{array}}$$
   and considering the projections given by $P_n(x,y) = (x,0)$ and $Q_n(x,y) =
   (0,y)$ we have
      $$ \|\cA_{m,n}P_n\|
         = \pfrac{\mu_{m-1}}{\mu_n}^a
            \pfrac{\mu_m}{\mu_{n-1}}^a \,
            \pfrac{\nu_{m-1}^{\cos((m-1)\pi)-1}}
            {\nu_{n-1}^{\cos((n-1)\pi)-1}}^{\eps/2}$$
   and
      $$ \|\cA_{m,n}^{-1} Q_m\|
         = \pfrac{\mu_{m}}{\mu_{m-1}}^{-b}
            \pfrac{\mu_{m-1}}{\mu_n}^{-b} \,
            \pfrac{\nu_{m-1}^{\cos((m-1)\pi)-1}}
            {\nu_{n-1}^{\cos((n-1)\pi)-1}}^{-\eps/2}$$
   and this implies
      $$ \|\cA_{m,n}P_n\|
         \le \pfrac{\mu_m}{\mu_{n-1}}^a \, \nu_{n-1}^\eps
         \ \ \ \text{ and } \ \ \
         \|\cA_{m,n}^{-1} Q_m\|
         \le \pfrac{\mu_{m-1}}{\mu_n}^{-b} \, \nu_{m-1}^\eps.$$
   This example shows that for every growth rates $\mu$ and $\nu$ we have a
   nonuniform $(\mu,\nu)$-dichotomy.

   Moreover, if $m$ is even, $n$ is odd and $\mu_m/\mu_{m-1}$ is bounded by a
   constant $\lbd$ then
      $$ \lbd^{-b} \pfrac{\mu_{m-1}}{\mu_n}^{-b} \, \nu_{m-1}^\eps
         \le \|\cA_{m,n}^{-1} Q_m\|
         \le \pfrac{\mu_{m-1}}{\mu_n}^{-b} \, \nu_{m-1}^\eps$$
   and this shows that the nonuniform part of the dichotomy can not be
   removed.
\end{example}

\begin{example}
   With $\mu_n = \nu_n = \e^n$ we get the local stable manifold theorem
   obtained by Barreira and Valls in \cite{Barreira-Valls-DCDS-2006}. Here,
   condition~\eqref{eq:CondicaoTeo} becomes $a + \eps < b$,
   condition~\eqref{int-conv} becomes $aq+\eps < 0$ and since
      $$ \beta_m
         = \e^{-a+\eps(1+1/q)} \prts{1-\e^{aq+\eps}}^{1/q}
            \e^{-\eps(1+2/q)m}$$
   and
      $$ \dfrac{\mu_m^a \beta_m^{-1}}{\mu_{n-1}^a \beta_n^{-1}}
         = \e^a \e^{\prts{a + \eps(1+2/q)}(m-n)},$$
         condition~\eqref{eq:decreasing}
   becomes $a + \eps(1 + 2/q) \le 0$, that is the condition $a + \beta \le 0$
   in~\cite{Barreira-Valls-DCDS-2006}. Note that condition $a + \eps(1 + 2/q)
   \le 0$ implies $a + \eps < b$ and $aq+\eps < 0$, although the first
   implication seems to have been unnoticed
   in~\cite{Barreira-Valls-DCDS-2006}.
\end{example}

\begin{example}
   We will now consider the polynomial case, i.e., $\mu_n = \nu_n =
   1+n$. For these rates, condition~\eqref{eq:CondicaoTeo} becomes $a + \eps <
   b$ and condition~\eqref{int-conv} becomes $aq+\eps +1 < 0$. Since
      $$ \int_m^{+\infty} (1+t)^{aq+\eps} \dt
         \le \sum_{k = m}^{+\infty} \mu_k^{aq} \nu_k^\eps
         = \sum_{k = m}^{+\infty} (1+k)^{aq +\eps}
         \le \int_{m-1}^{+\infty} (1+t)^{aq+\eps} \dt,$$
   we obtain the estimates
      $$ \dfrac{1}{\abs{aq+\eps+1}} (1+m)^{aq+\eps+1}
         \le \sum_{k = m}^{+\infty} (1+k)^{aq +\eps}
         \le \dfrac{1}{\abs{aq+\eps+1} 2^{aq+\eps+1}} (1+m)^{aq+\eps+1}$$
   and this implies
      $$ \beta_m
         \le \dfrac{\abs{aq+\eps+1}^{1/q}}{2^{a-\eps(1+1/q)}}
            (1+m)^{-\eps(1+2/q)-1/q}$$
   and
      $$ \beta_m
         \ge 2^{a+\eps/q+1/q} \abs{aq+\eps+1}^{1/q}
         (1+m)^{-\eps(1+2/q)-1/q}.$$
   Hence
      $$ \dfrac{\mu_m^a \beta_m^{-1}}{\mu_{n-1}^a \beta_n^{-1}}
         \le 2^{-2a+\eps-1/q} \pfrac{1+m}{1+n}^{a+\eps(1+2/q)+1/q}$$ and to
   have condition~\eqref{eq:decreasing} we need to have $a+\eps(1+2/q)+1/q \le
   0$. Therefore, taking into account that $a+\eps(1+2/q)+1/q \le 0$ implies
   $a + \eps < b$ and also implies $aq+\eps +1 < 0$ when $\eps > 0$, if
   $a+\eps(1+2/q)+1/q \le 0$ and $\eps > 0$ we have a local stable manifold
   theorem. If $aq + 1 < 0$ and $\eps = 0$ we also have a local stable manifold
   theorem.
\end{example}

\begin{example}
   In this example we will consider a nonuniform dichotomy with the following
   growth rates
      $$ \mu_n = (1+n) \prts{1 + \log(1+n)}^\lbd
         \ \ \ \text{ and } \ \ \
         \nu_n = 1 + \log(1+n),$$
   with $\lbd \ge 0$. Then condition~\eqref{eq:CondicaoTeo} is satisfied  for
   every $a < 0 \le b$ and every $\eps \ge 0$. The series in~\eqref{int-conv}
   becomes
      $$ \sum_{k=1}^{\infty} \mu_k^{aq} \nu_k^\eps
         = \sum_{k=1}^\infty (1+k)^{aq} \prts{1+\log(1+k)}^{\lbd aq + \eps}$$
   and is convergent if $aq < -1$ or if $aq=-1$ and $\eps - \lbd < -1$.

   If $aq < -1$, there are positive constants $\theta_1$ and $\theta_2$ such
   that
      $$ \sum_{k=m}^{\infty} \mu_k^{aq} \nu_k^\eps
         \ge \theta_1 \ (1+m)^{aq+1} \prts{1 + \log(1+m)}^{\lbd aq + \eps}$$
   and
      $$ \sum_{k=m}^{\infty} \mu_k^{aq} \nu_k^\eps
         \le \theta_2 \ (1+m)^{aq+1} \prts{1 + \log(1+m)}^{\lbd aq + \eps}$$
   for every $m \in \N$ and this implies that
      $$ \beta_m
         \le \dfrac{\theta_1^{-1/q}}{2^a (1+\log 2)^{\lbd a -\eps(1+1/q)}} \
            (1+m)^{- 1/q} \prts{1 + \log(1+m)}^{- \eps(1+2/q)}$$
   and
      $$ \beta_m \ge \theta_2^{-1/q} (1+m)^{- 1/q} \
            \prts{1 + \log(1+m)}^{- \eps(1+2/q)}$$
   for every $m \in \N$. Hence
      $$ \dfrac{\mu_m^a \beta_m^{-1}}{\mu_{n-1}^a \beta_n^{-1}}
         \le A \pfrac{m+1}{n+1}^{a+1/q}
         \pfrac{1+\log(1+m)}{1+\log(1+n)}^{\lbd a + \eps(1+2/q)}$$
   with
      $$ A = \dfrac{\theta_2^{1/q}}{\theta_1^{1/q} 2^a (1+ \log 2)^{\lbd a -
      \eps(1+1/q)}}$$
   and since $aq < -1$ condition~\eqref{eq:decreasing} is always satisfied.
   Therefore if $aq < -1$ we have a local stable manifold theorem for every
   nonuniform dichotomy with these growth rates.

   When $aq=-1$ and $\eps - \lbd < -1$, there are positive constants $\theta_3$
   and $\theta_4$ such that
      $$ \theta_3 \prts{1 + \log(1+m)}^{-\lbd + \eps + 1}
         \le \sum_{k=m}^{+\infty} \mu_k^{-1} \nu_k^\eps
         \le \theta_4 \prts{1 + \log(1+m)}^{-\lbd + \eps + 1}$$
   for every $m \in \N$. This estimates imply that
      $$ \beta_m
         \le \dfrac{\theta_3^{-1/q}}{2^a (1+\log 2)^{\lbd a -\eps(1+1/q)}} \
            (1+m)^{-1/q} \prts{1 + \log(1+m)}^{- \eps(1+2/q) - 1/q}$$
   and
      $$ \beta_m
         \ge \theta_4^{-1/q} (1+m)^{- 1/q} \
         \prts{1 + \log(1+m)}^{- \eps(1+2/q) - 1/q}$$
   for every $m \in \N$. Hence
      $$ \dfrac{\mu_m^a \beta_m^{-1}}{\mu_{n-1}^a \beta_n^{-1}}
         \le \dfrac{\theta_4^{1/q}}{\theta_3^{1/q} 2^a (1+ \log 2)^{- \lbd/q -
         \eps(1+1/q)}}
         \pfrac{1+\log(1+m)}{1+\log(1+n)}^{(1-\lbd)/q + \eps(1+2/q)}$$
   for every $m \in \N$ and condition~\eqref{eq:decreasing} is satisfied if
   $(1-\lbd)/q + \eps(1+2/q) \le 0$. Since $(1-\lbd)/q + \eps(1+2/q) \le 0$ and
   $\eps > 0$ imply $\eps - \lbd  < -1$, if $aq=-1$ and $(1-\lbd)/q +
   \eps(1+2/q) \le 0$ and $\eps > 0$ we have a local stable manifold theorem
   for nonuniform dichotomies with these rates. If $aq = -1$, $\eps = 0$ and
   $\lbd > 1$ we also have a local stable manifold theorem.
\end{example}
\section{Proof of Theorem \ref{thm:local}} \label{Section:Proof}
Given $n\in \N$ and $v_n=(\xi,\eta)\in E_n \times F_n$, using~\eqref{eq:traj},
it follows that for each $m > n$, the trajectory $\prts{v_m}_{m > n}$
satisfies the following equations
\begin{align}
   x_m & = \cA_{m,n} \xi + \sum_{k=n}^{m-1} \cA_{m,k+1} P_{k+1} f_k(x_k,y_k),
      \label{eq:dyn-split1a}\\
   y_m  &= \cA_{m,n} \eta +\sum_{k=n}^{m-1} \cA_{m,k+1} Q_{k+1} f_k(x_k,y_k).
      \label{eq:dyn-split1b}
\end{align}
In view of the forward invariance mentioned in \eqref{thm:local:invar}, each
trajectory of~\eqref{eq:dyn} starting in
$\cV_{\phi,n,\delta/(CK),\tilde{\beta}}$ must be
in $\cV_{\phi,m,\delta,\beta}$ for every $m \ge n$, and thus the
equations~\eqref{eq:dyn-split1a} and~\eqref{eq:dyn-split1b} can be written in
the form
\begin{align}
   & x_m = \cA_{m,n} \xi + \sum_{k=n}^{m-1} \cA_{m,k+1} P_{k+1}
      f_k(x_k,\phi_k(x_k)), \label{eq:dyn-split2a}\\
   & \phi_m (x_m) = \cA_{m,n} \phi_n(\xi) + \sum_{k=n}^{m-1} \cA_{m,k+1}
      Q_{k+1} f_k (x_k,\phi_k(x_k)). \label{eq:dyn-split2b}
\end{align}
To prove that equations~\eqref{eq:dyn-split2a} and~\eqref{eq:dyn-split2b} have
solutions we will use Banach fixed point theorem in some suitable complete
metric spaces.

In $\cX_{\delta,\beta}$ we define a metric by
\begin{equation}\label{def:metric:X_beta}
   \|\phi - \psi\|'
   = \sup\set{\dfrac{\|\phi_n(\xi) - \psi_n(\xi)\|}{\|\xi\|} : n \in \N
         \text{ and } \xi \in B_n(\delta \beta_n)\setminus \set{0}}.
\end{equation}
for each $\phi=(\phi_n)_{n \in \N}$, $\psi=(\psi_n)_{n \in \N} \in
\cX_{\delta,\beta}$.
It is easy to see that $\cX_{\delta,\beta}$ is a complete metric space with the
metric
defined by~\eqref{def:metric:X_beta}.

We also need to consider the space $\cX^*_{\delta,\beta}$ of sequences
$\phi=(\phi_n)_{n \in
\N}$ with $\phi_n\colon E_n \to F_n$ such that the sequence
$\prts{\phi_n|B_n(\delta \beta_n)}_{n \in \N}$ is in $\cX_{\delta,\beta}$ and,
for each
$n \in \N$,
   $$ \phi_n(\xi)=\phi_n\prts{\dfrac{\delta \beta_n \xi}{\|\xi\|}}
      \text{ whenever $\xi \not \in B_n(\delta \beta_n)$.}$$ There is a
one-to-one correspondence between sequences in $\cX_{\delta,\beta}$ and in
$\cX^*_{\delta,\beta}$ because for each sequence of functions
$\phi=\prts{\phi_n}_{n \in
\N} \in \cX_{\delta,\beta}$ there is a unique extension $\widetilde{\phi} =
\prts{\widetilde{\phi}_n}_{n \in \N}$ such that each $\widetilde{\phi}_n$ is a
Lipschitz extension of $\phi_n$ to $\overline{B_n(\delta \beta_n)}$. This
one-to-one correspondence allows to define a metric in $\cX_{\delta,\beta}^*$.
For every
$\phi=\prts{\phi_n}_{n \in \N}, \psi=\prts{\phi_n}_{n \in \N} \in
\cX_{\delta,\beta}^*$,
we define this metric by
   $$ \|\phi - \psi\|' = \|\overline\phi - \overline\psi\|'$$
where $\overline\phi=\prts{\phi_n|_{B_n(\delta \beta_n)}}_{n \in \N}$,
$\overline\psi=\prts{\psi_n|_{B_n(\delta \beta_n)}}_{n \in \N}$ and the right
hand side is the metric defined by~\eqref{def:metric:X_beta}. Is is easy to
see
that with this metric $\cX^*_{\delta,\beta}$ is a complete metric space.

Furthermore, given $\phi=\prts{\phi_n}_{n \in \N}, \psi=\prts{\psi_n}_{n \in
\N} \in\cX_{\delta,\beta}^*$, one can easily verify that
\begin{align}
   & \|\phi_n(\xi)-\phi_n(\bar\xi)\| \le 2 \|\xi-\bar\xi\|
      \label{cond-phi-e-1}\\
   & \|\phi_n(\xi) - \psi_n(\xi)\| \le \|\phi - \psi\|' \|\xi\|
      \label{cond-phi-e-2}
\end{align}
for every $n \in \N$ and every $\xi, \bar\xi \in E_n$.

Let $\cB = \cB_{n,\delta,\beta}$ be the space of all sequences
$x=\prts{x_m}_{m\ge n}$
of functions
   $$ x_m \colon B_n(\delta \beta_n) \to E_m$$
such that
\begin{align}
   & x_n (\xi) = \xi, \ \ \ x_m(0) =0 \label{cond-x_m-0}\\
   & \|x_m(\xi) - x_m(\bar\xi)\|
         \le C \pfrac{\mu_m}{\mu_{n-1}}^a \nu_{n-1}^\eps \ \|\xi-\bar\xi\|
         \label{cond-x_m-1}
\end{align}
for every $m \ge n$ and every $\xi, \bar{\xi} \in B_n(\delta \beta_n)$. Making
$\bar\xi=0$ in~\eqref{cond-x_m-1} we obtain the following estimates
\begin{equation} \label{cond-x_m-1a}
   \|x_m(\xi)\|
   \le C \pfrac{\mu_m}{\mu_{n-1}}^a \nu_{n-1}^\eps \|\xi\|
   \le C \delta \pfrac{\mu_m}{\mu_{n-1}}^a \nu_{n-1}^\eps  \beta_n
\end{equation}
for every $m \ge n$ and every $\xi\in B_n(\delta \beta_n)$. In
$\cB_{n,\delta,\beta}$
we define a metric by
\begin{equation} \label{def:metric-of-B}
   \|x-y\|''
   = \sup \set{\dfrac{\|x_m (\xi) - y_m(\xi)\|}{\|\xi\|}
      \pfrac{\mu_m}{\mu_{n-1}}^{-a} \nu_{n-1}^{-\eps} \colon
      m \ge n, \ \xi \in B_n(\delta \beta_n)}
\end{equation}
for every $x$, $y \in \cB_{n,\delta,\beta}$. It is easy to see that with this
metric $\cB_{n,\delta,\beta}$ is a complete metric space.

\begin{lemma}\label{lemma:Exist-Suc-x_m}
   Given $\delta >0$ sufficiently small, for each $\phi\in
   \cX^*_{\delta,\beta}$ and $n \in \N$ there exists a unique sequence
   $x=x^{\phi}\in \cB_{n,\delta,\beta}$ satisfying the
   equation~\eqref{eq:dyn-split2a} for every $m \ge n$ and $\xi\in B_n(\delta
   \beta_n)$. Moreover, choosing $\delta > 0$ sufficiently small, we have
   \begin{equation} \label{ine:norm:x^phi-x^psi}
      \|x^\phi - x^\psi\|''
      \le \dfrac{C}{3} \nu_{n-1}^{-\eps} \|\phi - \psi\|'
   \end{equation}
   for each $\phi, \psi \in \cX^*_{\delta,\beta}$.
\end{lemma}

\begin{proof}
   Given $\phi \in \cX_{\delta,\beta}^*$, we define an operator $J = J_\phi$
   in~$\cB_{n,\delta,\beta}$ by
   \begin{equation} \label{def:J}
      (Jx)_m(\xi)=
      \begin{cases}
         \xi & \text{ if } m = n,\\
         \cA_{m,n} \xi + \dsum_{k=n}^{m-1} \cA_{m,k+1} P_{k+1}
            f_k(x_k(\xi),\phi_k(x_k(\xi))) & \text{ if } m > n.
      \end{cases}
   \end{equation}
   One can easily verify from~\eqref{cond-x_m-0}, \eqref{cond-phi-0} and
   \eqref{cond-f-0} that $(Jx)_m(0)=0$ for every $m \ge n$.

   Let $x \in \cB_{n,\delta,\beta}$ and, for every $k \ge n$, put
      $$ \alpha_k  = \|f_k(x_k(\xi),\phi_k(x_k(\xi)))
         - f_k(x_k(\bar\xi),\phi_k(x_k(\bar\xi)))\|$$
   with $\xi, \bar\xi \in B_n(\delta \beta_n)$. From~\eqref{def:J} it follows
   that
   \begin{equation} \label{eq:identifica}
      \|(Jx)_m (\xi) - (Jx)_m (\bar\xi)\|
      \le \| \cA_{m,n}P_n\| \, \|\xi - \bar\xi\|
         + \sum_{k=n}^{m-1} \|\cA_{m,k+1} P_{k+1}\| \, \alpha_k
   \end{equation}
   for every $m > n$.
   From~\eqref{cond-f-1},~\eqref{cond-phi-e-1},~\eqref{cond-x_m-1}
   and~\eqref{cond-x_m-1a} we obtain
   \begin{equation}\label{ine:norm:f_k(xi)-f_k(bar_xi)}
      \begin{split}
         \alpha_k
         & \le c \prts{\|x_k(\xi) - x_k(\bar\xi)\|
            + \|\phi_k(x_k(\xi)) - \phi_k(x_k(\bar\xi))\|} \times \\
         & \qquad \times \prts{\|x_k(\xi)\|
            + \|\phi_k(x_k(\xi))\| + \| x_k(\bar\xi)\|
            + \|\phi_k(x_k(\bar\xi))\|}^q\\
         & \le 3^{q+1} c \|x_k(\xi) - x_k(\bar\xi)\|
            \prts{\|x_k(\xi)\| + \|x_k(\bar\xi)\|}^q\\
         & \le c (3C)^{q+1} \prts{2 \delta}^q
            \prts{\dfrac{\mu_k}{\mu_{n-1}}}^{aq+a}
             \nu_{n-1}^{\eps (q+1)} \beta_n^q
            \|\xi - \bar\xi\|.
      \end{split}
   \end{equation}
   By~\eqref{def:beta} we get
   \begin{equation} \label{eq:b_k,mu_k,nu_k=1}
      \mu_{n-1}^{-aq} \nu_{n-1}^{\eps (q+1)} \beta_n^q
      \sum_{k=n}^{+\infty} \mu_k^{aq} \nu_k^\eps=1
   \end{equation}
   and this together with ~\eqref{ine:norm:f_k(xi)-f_k(bar_xi)}
   and~\eqref{eq:dich-1} imply
   \begin{align*}
      & \sum_{k=n}^{m-1} \| \cA_{m,k+1} P_{k+1}\| \, \alpha_k\\
      & \le c (3C)^{q+1} D (2 \delta)^q \pfrac{\mu_m}{\mu_{n-1}}^a
         \|\xi - \bar\xi\|
         \mu_{n-1}^{-aq} \nu_{n-1}^{\eps (q+1)} \beta_n^q
         \sum_{k=n}^{m-1} \mu_k^{aq} \nu_k^\eps\\
      & \le c (3C)^{q+1} D (2 \delta)^q \pfrac{\mu_m}{\mu_{n-1}}^a
         \|\xi - \bar\xi\|.
   \end{align*}
   From last estimate,~\eqref{eq:identifica} and~\eqref{eq:dich-1} we have
   \begin{eqnarray*}
      && \|(Jx)_m (\xi) - (Jx)_m (\bar\xi)\|\\
      && \le D \pfrac{\mu_m}{\mu_{n-1}}^a \nu_{n-1}^\eps \|\xi - \bar\xi\|
         + c (3C)^{q+1} D (2 \delta)^q \pfrac{\mu_m}{\mu_{n-1}}^a
         \|\xi - \bar\xi\|
   \end{eqnarray*}
   for every $m \ge n$  and every $\xi, \bar\xi \in B_n(\delta \beta_n)$. Since
   $C > D$, choosing $\delta$ sufficiently small we obtain
      $$ \|(Jx)_m (\xi) - (Jx)_m (\bar\xi)\|
         \le C \pfrac{\mu_m}{\mu_{n-1}}^a \nu_{n-1}^\eps \|\xi - \bar\xi\|$$
   and this implies the inclusion
   $J(\cB_{n,\delta,\beta})\subset\cB_{n,\delta,\beta}$.

   We now show that $J$ is a contraction for the metric induced
   by~\eqref{def:metric-of-B}. Let $x,y\in \cB_{n,\delta,\beta}$. Then
   \begin{equation}\label{ine:norm:J_x_m(xi)-J_y_m(bar_xi)}
      \begin{split}
         & \|(Jx)_m(\xi)-(Jy)_m(\xi)\|\\
         & \le \sum_{k=n}^{m-1} \|\cA_{m,k+1}P_{k+1}\| \
            \|f_k(x_k(\xi),\phi_k(x_k(\xi))) -
            f_k(y_k(\xi),\phi_k(y_k(\xi)))\|
      \end{split}
   \end{equation}
   for every $m \ge n$ and every $\xi, \in B_n(\delta \beta_n)$. By
   \eqref{cond-f-1},~\eqref{cond-phi-e-1},~\eqref{def:metric-of-B}
   and~\eqref{cond-x_m-1a} we have for every $k \ge n$
   \begin{equation}\label{ine:norm:f_k(x_k)-f_k(y_k)}
      \begin{split}
         & \|f_k(x_k(\xi),\phi_k(x_k(\xi))) -
         f_k(y_k(\xi),\phi_k(y_k(\xi)))\|\\
         & \le c \prts{\| x_k(\xi) - y_k(\xi)\| + \|\phi_k(x_k(\xi)) -
         \phi_k(y_k(\xi))\|} \times\\
         & \qquad\qquad\qquad\qquad
            \times \prts{\|x_k(\xi)\|+ \|\phi_k(x_k(\xi))\| + \|y_k(\xi)\|
            + \|\phi_k(y_k(\xi))\|}^q\\
         & \le 3^{q+1} c \| x_k(\xi) - y_k(\xi)\|
            \prts{\|x_k(\xi)\|+ \|y_k(\xi)\|}^q\\
         & \le  3^{q+1} c
            \prts{\dfrac{\mu_k}{\mu_{n-1}}}^a \nu_{n-1}^\eps
            \| x - y\|'' \|\xi\|
            (2 C \delta)^q \prts{\dfrac{\mu_k}{\mu_{n-1}}}^{aq}
            \nu_{n-1}^{\eps q} \beta_n^q \\
         & \le 3^{q+1} c (2 C \delta)^q \prts{\dfrac{\mu_k}{\mu_{n-1}}}^{aq+a}
               \nu_{n-1}^{\eps (q+1)} \beta_n^q \| x - y\|'' \|\xi\|.
      \end{split}
   \end{equation}
   Hence, from \eqref{ine:norm:J_x_m(xi)-J_y_m(bar_xi)},~\eqref{eq:dich-1}
   and~\eqref{ine:norm:f_k(x_k)-f_k(y_k)} we have
   \begin{align*}
      & \|(Jx)_m(\xi) -(Jy)_m(\xi)\|\\
      & \le 3^{q+1} c (2C\delta)^q D \pfrac{\mu_m}{\mu_{n-1}}^a
         \|\xi\| \| x - y\|'' \mu_{n-1}^{-aq} \nu_{n-1}^{\eps (q+1)} \beta_n^q
         \sum_{k=n}^{m-1} \mu_k^{aq} \nu_n^\eps \\
      & \le 3^{q+1} c (2C\delta)^q D \pfrac{\mu_m}{\mu_{n-1}}^a
         \|\xi\| \| x - y\|''
   \end{align*}
   for every $m \ge n$ and every $\xi \in B_n(\delta \beta_n)$ and this
   implies
      $$ \|Jx-Jy\|'' \le 3^{q+1} c (2C\delta)^q D \|x-y\|''.$$
   Choosing $\delta > 0$ such that $3^{q+1} c (2C\delta)^q D < 1$ it
   follows that $J$ is a contraction in $\cB_{n,\delta,\beta}$. Because
   $\cB_{n,\delta,\beta}$ is complete, by
   the Banach fixed point theorem, the map $J$ has a unique fixed point
   $x^\phi$ in~$\cB_{n,\delta,\beta}$, which is thus the desired sequence.

   Next we will prove~\eqref{ine:norm:x^phi-x^psi}. Let $\phi, \psi \in
   \cX_{\delta,\beta}^*$. From~\eqref{eq:dyn-split2a} we have
   \begin{equation}\label{ine:norm:x_m^phi(xi)-x_m^psi(xi)}
      \begin{split}
         & \|x_m^\phi(\xi) - x_m^\psi(\xi)\|\\
         & \le \dsum_{k=n}^{m-1} \|\cA_{m,k+1} P_{k+1}\| \,
            \|f_k(x^\phi_k(\xi),\phi_k(x^\phi_k(\xi)))
            - f_k(x^\psi_k(\xi),\psi_k(x^\psi_k(\xi)))\|
      \end{split}
   \end{equation}
   for every $m \ge n$ and every $\xi \in B_n(\delta \beta_n)$.
   By~\eqref{cond-f-1},~\eqref{cond-phi-e-1},~\eqref{def:metric-of-B},
   \eqref{cond-phi-e-2} and~\eqref{cond-x_m-1a} it follows that
   \begin{equation}\label{ine:norm:f_k(x^phi_k...)- f_k(x^psi_k...)}
      \begin{split}
         & \|f_k(x^\phi_k(\xi),\phi_k(x^\phi_k(\xi)))
               - f_k(x^\psi_k(\xi),\psi_k(x^\psi_k(\xi)))\|\\
         & \le c \prts{\|x^\phi_k(\xi) - x^\psi_k(\xi)\|
            + \|\phi_k (x^\phi_k(\xi)) - \psi_k(x^\psi_k(\xi))\|} \times\\
         & \qquad\qquad\qquad
            \times \prts{\|x^\phi_k(\xi)\| + \|\phi_k (x^\phi_k(\xi))\|
            +\|x^\psi_k(\xi)\| + \|\psi_k(x^\psi_k(\xi))\|}^q\\
         & \le c \prts{3 \|x^\phi_k(\xi) - x^\psi_k(\xi)\|
            + \|\phi_k (x^\psi_k(\xi)) - \psi_k(x^\psi_k(\xi))\|}
            \prts{3 \|x^\phi_k(\xi)\| + 3 \|x^\psi_k(\xi)\|}^q\\
         & \le c (6C\delta)^q \prts{3 \|x^\phi - x^\psi\|''
            \pfrac{\mu_k}{\mu_{n-1}}^{a} \nu_{n-1}^\eps \|\xi\|
            + \|\phi - \psi\|' \|x^\psi_k(\xi)\|} \times\\
         &  \qquad\qquad\qquad\times
            \pfrac{\mu_k}{\mu_{n-1}}^{aq} \nu_{n-1}^{\eps q} \beta_n^q\\
         & \le c (6C\delta)^q
            \prts{3 \|x^\phi - x^\psi\|'' + C \|\phi - \psi\|'} \|\xi\|
            \pfrac{\mu_k}{\mu_{n-1}}^{aq+a} \nu_{n-1}^{\eps (q+1)} \beta_n^q
      \end{split}
   \end{equation}
   for every $k \ge n$. Hence by~\eqref{ine:norm:x_m^phi(xi)-x_m^psi(xi)},
   last
   inequality,~\eqref{eq:dich-1} and~\eqref{eq:b_k,mu_k,nu_k=1} we get
   \begin{align*}
      \|x_m^\phi(\xi) - x_m^\psi(\xi)\|
      & \le c (6C\delta)^q  D
         \prts{3 \|x^\phi - x^\psi\|'' + C \|\phi - \psi\|'} \times\\
      & \qquad\qquad \times \|\xi\|
         \pfrac{\mu_m}{\mu_{n-1}}^a \nu_{n-1}^\eps
         \mu_{n-1}^{-aq} \nu_{n-1}^{\eps q} \beta_n^q
         \dsum_{k=n}^m \mu_k^{aq} \nu_k^\eps\\
      & \le c (6C\delta)^q  D
         \prts{3 \|x^\phi - x^\psi\|'' + C \|\phi - \psi\|'} \|\xi\|
         \pfrac{\mu_m}{\mu_{n-1}}^a
   \end{align*}
   for every $m \ge n$ and every $\xi \in B_n(\delta \beta_n)$ and this
   implies
      $$ \|x^\phi - x^\psi\|''
         \le c (6C\delta)^q  D \nu_{n-1}^{-\eps}
         \prts{3 \|x^\phi - x^\psi\|'' + C \|\phi - \psi\|'}.$$
   Choosing $\delta > 0$ such that $c (6C\delta)^q  D < 1/6$ we
   have~\eqref{ine:norm:x^phi-x^psi}.
\end{proof}

We now represent by $\prts{x_{n,k}^\phi}_{k \ge n} \in \cB_{n,\delta,\beta}$
the unique sequence given by Lemma~\ref{lemma:Exist-Suc-x_m}.

\begin{lemma} \label{lemma:equiv}
   Given $\delta>0$ sufficiently small and $\phi \in \cX^*_{\delta,\beta}$ the
   following
   properties hold:
   \begin{enumerate}[$1)$]
      \item If for every $n \in \N$, $m \ge n$ and $\xi \in B_n(\delta
          \beta_n)$ the identity~\eqref{eq:dyn-split2b} holds with $x =
          x^\phi$, then
          \begin{equation} \label{eq:phi_n}
            \phi_n(\xi)
               = - \sum_{k=n}^\infty \cA_{k+1,n}^{-1}
                  Q_{k+1} f_k(x^\phi_{n,k}(\xi),
                  \phi_k(x^\phi_{n,k}(\xi))).
          \end{equation}
         for every $n \in \N$ and every $\xi \in B_n(\delta \beta_n)$.
      \item If for every $n \in \N$ and every $\xi \in B_n(\delta
          \beta_n)$ the equation~\eqref{eq:phi_n} holds, then
          \eqref{eq:dyn-split2b} holds  with $x = x^\phi$ for every $n \in
          \N$, every $m \ge n$ and every $\xi \in
          B_n(\delta\tilde\beta_n/(CK))$.
   \end{enumerate}
\end{lemma}

\begin{proof}
   First we prove that the series in~\eqref{eq:phi_n} is convergent.
   From~\eqref{eq:dich-2},~\eqref{cond-f-1a},~\eqref{cond-phi-e-1}
   and~\eqref{cond-x_m-1a}, we conclude that for every $n \in \N$ and every
   $\xi \in B_n(\delta \beta_n)$
   \begin{align*}
      & \sum_{k=n}^{\infty} \|\cA_{k+1,n}^{-1} Q_{k+1}
         f_k(x_{n,k}^{\phi}(\xi),\phi_k(x_{n,k}^{\phi}(\xi)))\|\\
      & \le \sum_{k=n}^{\infty} \|\cA_{k+1,n}^{-1} Q_{k+1}\| \
         \|f_k(x_{n,k}^{\phi}(\xi),\phi_k(x_{n,k}^{\phi}(\xi)))\| \\
      & \le \sum_{k=n}^{\infty} D \pfrac{\mu_k}{\mu_n}^{-b} \nu_k^\eps
         c \prts{\|x^\phi_{n,k}(\xi)\| +
         \|\phi_k(x^\phi_{n,k}(\xi))\|}^{q+1}\\
      & \le c D  \sum_{k=n}^{\infty} \pfrac{\mu_k}{\mu_n}^{-b} \nu_k^\eps
         \prts{3 C \delta \pfrac{\mu_k}{\mu_{n-1}}^a \nu_{n-1}^{\eps}
         \beta_n}^{q+1}\\
      & \le c(3C\delta)^{q+1} D  \mu_n^b \mu_{n-1}^{-aq-a}
         \nu_{n-1}^{\eps(q+1)}
         \beta_n^{q+1} \sum_{k=n}^{\infty} \mu_k^{aq+a-b} \nu_k^\eps \\
      & \le c(3C\delta)^{q+1} D  \mu_{n-1}^{-aq} \nu_{n-1}^{\eps(q+1)}
         \beta_n^{q+1} \sum_{k=n}^{\infty} \mu_k^{aq} \nu_k^\eps \\
      & \le c(3C\delta)^{q+1} D \beta_n
   \end{align*}
   and thus the series converges.

   Now, let us suppose that~\eqref{eq:dyn-split2b} holds with $x = x^\phi$ for
   every $n \in \N$, every $m \ge n$ and every $\xi \in B_n(\delta \beta_n)$.
   Then, since $\cA_{m,n}^{-1}\cA_{m,k+1}=\cA_{k+1,n}^{-1}$ for $n \le k \le
   m-1$, equation~\eqref{eq:dyn-split2b} can be written in the following
   equivalent
   form
   \begin{equation} \label{eq:equiv1A}
      \phi_{n}(\xi) = \cA_{m,n}^{-1} \phi_m(x_{n,m}^{\phi}(\xi)) -
      \sum_{k=n}^{m-1} \cA_{k+1,n}^{-1}
      Q_{k+1} f_k(x_{n,k}^{\phi}(\xi),\phi_k(x_{n,k}^{\phi}(\xi))).
   \end{equation}
   Using~\eqref{eq:dich-2},~\eqref{cond-phi-e-1} and~\eqref{cond-x_m-1a},
   we have
   \begin{align*}
      \|\cA_{m,n}^{-1} \phi_m(x_{n,m}^{\phi}(\xi))\|
      & = \|\cA_{m,n}^{-1} Q_m \phi_m(x_{n,m}^\phi(\xi))\| \\
      & \le 2 D \pfrac{\mu_{m-1}}{\mu_n}^{-b} \nu_{m-1}^\eps
         \|x_{n,m}^\phi(\xi)\|\\
      & \le 2 D \pfrac{\mu_{m-1}}{\mu_n}^{-b} \nu_{m-1}^\eps
         C \delta \pfrac{\mu_m}{\mu_{n-1}}^a \nu_{n-1}^\eps \beta_n \\
      & \le 2 C D \delta \mu_m^a \mu_{m-1}^{-b} \nu_{m-1}^\eps
         \mu_n^b \mu_{n-1}^{-a} \nu_{n-1}^\eps \beta_n
   \end{align*}
   and by~\eqref{eq:CondicaoTeo} this converge to zero when $m\to\infty$.
   Hence, letting $m\to\infty$ in \eqref{eq:equiv1A} we obtain the
   identity~\eqref{eq:phi_n} for every $n \in \N$ and every $\xi \in B_n(\delta
   \beta_n)$.

   We now assume that for every $n \in \N$, $m \ge n$ and $\xi \in B_n(\delta
   \beta_n)$ the identity~\eqref{eq:phi_n} holds. If $\xi \in B_n(\delta
   \tilde\beta_n/(CK))$, then by \eqref{eq:decreasing} we get
   \begin{equation}\label{eq:norm:x_n,m(xi)}
      \|x_{n,m}(\xi)\|
      \le C \pfrac{\mu_m}{\mu_{n-1}}^a \nu_{n-1}^\eps
         \dfrac{\delta}{CK}\tilde \beta_n
      = \dfrac{\delta}{K} \dfrac{\mu_m^a \beta_m^{-1}}{\mu_{n-1}^a
         \beta_n^{-1}} \beta_m
      \le \delta \beta_m.
   \end{equation}
   Therefore
      $$ \cA_{m,n} \phi_n(\xi)
         = - \sum_{k=n}^{\infty} \cA_{m,n} \cA_{k+1,n}^{-1}
         Q_{k+1} f_k(x_{n,k}^{\phi}(\xi),\phi_k(x_{n,k}^{\phi}(\xi))),$$ and
   thus it follows from \eqref{eq:phi_n} and the uniqueness of the sequences
   $x^\phi$ that
   \begin{align*}
      \cA_{m,n} \phi_n(\xi) & + \sum_{k=n}^{m-1} \cA_{m,k+1}
         Q_{k+1} f_k(x_{n,k}^{\phi}(\xi),\phi_k(x_{n,k}^{\phi}(\xi))) \\
      & = - \sum_{k=m}^{\infty} \cA_{k+1,m}^{-1}
            Q_{k+1} f_k(x_{n,k}^{\phi}(\xi),\phi_k(x_{n,k}^{\phi}(\xi)))\\
      & = - \sum_{k=m}^{\infty} \cA_{k+1,m}^{-1}
            Q_{k+1} f_k(x_{m,k}^{\phi}(x_{n,m}^{\phi}(\xi)),
            \phi_k(x_{m,k}^{\phi}(x_{n,m}^{\phi}(\xi))))\\
      & = \phi_m(x_{n,m}^{\phi}(\xi))
   \end{align*}
   for every $n \in \N$, every $m \ge n$ and every $\xi \in B_n(\delta
   \beta_n/(CK))$. This proves the lemma.
\end{proof}

\begin{lemma} \label{lemma:Exist-Suc-phi}
   Given $\delta > 0$ sufficiently small there is a unique $\phi \in
   \cX^*_{\delta,\beta}$ such that
      $$ \phi_n(\xi)
         = - \sum_{k=n}^\infty \cA_{k+1,n}^{-1}
            Q_{k+1} f_k(x^\phi_k(\xi), \phi_k(x^\phi_k(\xi)))$$
   for every $n \in \N$ and every $\xi \in B_n(\delta \beta_n)$.
\end{lemma}

\begin{proof}
   We consider the operator $\Phi$ defined for each $\phi\in
   \cX^*_{\delta,\beta}$ by
   \begin{equation} \label{eq:op-Phi}
      (\Phi \phi)_n (\xi) =
      \begin{cases}
         - \dsum_{k=n}^{\infty} \cA_{k+1,n}^{-1} Q_{k+1} f_k
            (x_k^{\phi}(\xi), \phi_k(x_k^{\phi}(\xi)))
            & \text{ if } \xi \in B_n(\delta \beta_n),\\
         (\Phi \phi)_n (\delta \beta_n \xi/\|\xi\|)
            & \text{ if } \xi \not\in B_n(\delta \beta_n),
      \end{cases}
   \end{equation}
   where $x^{\phi}=(x^\phi_k)_{k\ge n} \in \cB_{n,\delta,\beta}$ is the unique
   sequence given by Lemma~\ref{lemma:Exist-Suc-x_m}. It follows
   from~\eqref{cond-f-0},~\eqref{cond-x_m-0},~\eqref{cond-phi-0}
   and~\eqref{eq:op-Phi} that $(\Phi\phi)_n(0)=0$ for each $n\in\N$.

   Furthermore, given $n \in \N$ and $\xi, \bar\xi \in B_n(\delta \beta_n)$,
   by~\eqref{eq:dich-2},~\eqref{ine:norm:f_k(xi)-f_k(bar_xi)}
   and~\eqref{eq:b_k,mu_k,nu_k=1} we have
   \begin{align*}
      & \|(\Phi \phi)_n (\xi) - (\Phi \phi)_n (\bar\xi)\| \\
      & \le \sum_{k=n}^{\infty} \|\cA_{k+1,n}^{-1} Q_{k+1}\|
         \cdot \|f_k(x^\phi_k(\xi),\phi_k(x^\phi_k(\xi)))
         - f_k(x^\phi_k(\bar\xi),\phi_k(x^\phi_k(\bar\xi)))\|\\
      & \le c (3C)^{q+1} D \prts{2 \delta}^q \|\xi - \bar\xi\|
         \mu_{n-1}^{-aq-a} \mu_n^b \nu_{n-1}^{\eps (q+1)} \beta_n^q
         \sum_{k=n}^{\infty} \mu_k^{aq+a-b} \nu_k^\eps\\
      & \le c (3C)^{q+1} D \prts{2 \delta}^q \|\xi - \bar\xi\|
         \mu_{n-1}^{-aq} \nu_{n-1}^{\eps (q+1)} \beta_n^q
         \sum_{k=n}^{\infty} \mu_k^{aq} \nu_k^\eps\\
      & = c (3C)^{q+1} D \prts{2 \delta}^q \|\xi - \bar\xi\|.
   \end{align*}
   Hence, choosing $\delta>0$ (independently of $\phi$, $n$
   and $\xi$) such that $ c (3C)^{q+1} D \prts{2 \delta}^q \le 1$  we have
      $$ \|(\Phi \phi)_n (\xi) - (\Phi \phi)_n (\bar\xi)\|
         \le \|\xi - \bar\xi\|.$$
   Therefore $\Phi(\cX^*_{\delta,\beta})\subset\cX^*_{\delta,\beta}$.

   We now show that $\Phi$ is a contraction. Given
   $\phi,\psi\in\cX^*_{\delta,\beta}$
   and $n\in \N$, let $x^{\phi}$ and $x^{\psi}$ be the unique sequences given
   by Lemma~\ref{lemma:Exist-Suc-x_m} respectively for $\phi$ and $\psi$.
   By~\eqref{eq:dich-2},~\eqref{ine:norm:f_k(x^phi_k...)-
   f_k(x^psi_k...)},~\eqref{ine:norm:x^phi-x^psi}
   and~\eqref{eq:b_k,mu_k,nu_k=1} we have
   \begin{align*}
      & \|(\Phi \phi)_n (\xi)-(\Phi \psi)_n(\xi)\| \\
      & \le \sum_{k=n}^{\infty} \|\cA_{k+1,n}^{-1} Q_{k+1}\|
         \|f_k(x^\phi_k(\xi),\phi_k(x^\phi_k(\xi)))
         - f_k(x^\psi_k(\xi),\phi_k(x^\psi_k(\xi)))\|\\
      & \le c D (6C\delta)^q
         \prts{3 \|x^\phi - x^\psi\|'' + C \|\phi - \psi\|'} \|\xi\|
         \mu_{n-1}^{-aq-a} \mu_n^b \nu_{n-1}^{\eps (q+1)} \beta_n^q
         \sum_{k=n}^{\infty} \mu_k^{aq+a-b} \nu_k^\eps\\
      & \le 2 c C^{q+1} D (6 \delta)^q  \|\xi\| \|\phi - \psi\|'
         \mu_{n-1}^{-aq} \nu_{n-1}^{\eps(q+1)} \beta_n^q
         \sum_{k=n}^{\infty} \mu_k^{aq} \nu_k^\eps \\
      & = 2 c C^{q+1} D (6 \delta)^q  \|\xi\| \|\phi - \psi\|'
   \end{align*}
   for every $n \in \N$ and every $\xi \in B_n(\delta \beta_n)$
   and this implies
      $$ \|\Phi \phi- \Phi \psi\|'
         \le 2 c C^{q+1} D (6 \delta)^q \|\phi - \psi\|'.$$ Choosing $\delta >
   0$ such that $2 c C^{q+1} D (6 \delta)^q<1$ it follows that $\Phi$ is a
   contraction in $\cX^*_{\delta,\beta}$. Therefore the map $\Phi$ has a unique
   fixed
   point $\phi$ in~$\cX^*_{\delta,\beta}$ that is the    desired sequence.
\end{proof}

We are now in conditions to prove Theorem~\ref{thm:local}.

\begin{proof}[Proof of Theorem~$\ref{thm:local}$]
   By Lemma~\ref{lemma:Exist-Suc-x_m}, for each $\phi\in \cX^*_{\delta,\beta}$
   there is
   a unique sequence $x^{\phi}\in\cB_{n,\delta,\beta}$
   satisfying~\eqref{eq:dyn-split2a}. It
   remains to show that there is a $\phi$ and a corresponding $x^\phi$ that
   satisfie~\eqref{eq:dyn-split2b}. By Lemma~\ref{lemma:equiv}, this is
   equivalent to solve~\eqref{eq:phi_n}. Finally, by
   Lemma~\ref{lemma:Exist-Suc-phi}, there is a unique solution of
   \eqref{eq:phi_n}. This establishes the existence of the stable manifolds for
   $\delta>0$ sufficiently small. Moreover, for each $n\in \N$, $m\ge n$ and
   $\xi,\bar{\xi}\in B_n(\delta \tilde\beta_n/(CK))$ it follows
   from~\eqref{eq:norm:x_n,m(xi)} and~\eqref{cond-phi-1} that
   \begin{align*}
      & \|\cF_{m,n}(\xi,\phi_n(\xi)) - \cF_{m,n}(\xi,\phi_n(\bar\xi))\|\\
      & \le \|x_m(\xi) - x_m(\bar\xi)\|
         + \|\phi_m(x_m(\xi)) - \phi_m(x_m(\bar\xi))\|\\
      & \le 2 \|x_m(\xi) - x_m(\bar\xi)\|\\
      & \le 2 C \pfrac{\mu_m}{\mu_{n-1}}^a \nu_{n-1}^{\eps} \|\xi-\bar \xi\|.
   \end{align*}
   Hence we obtain~\eqref{thm:ine:norm:F_mn(xi...)-F_mn(barxi...)} and the
   theorem is proved.
\end{proof}
\section{Behavior under perturbations}\label{section:perturbations}
In this section we assume that equation~\eqref{eq:lin:dif} admits a
$(\mu,\nu)$-dichotomy for some $D \ge 1$, $a <  0\le b$ and $\eps \ge 0$. Given
$c > 0$ and $q > 1$, let $\cP_{c,q}$ be the class of all sequences of function
$f=\prts{f_n}_{n \in \N}$ such that $f_n \colon X \to X$ and verify
conditions~\eqref{cond-f-0} and~\eqref{cond-f-1} with the given $c$ and $q$. In
$\cP_{c,q}$ we can define a metric by
\begin{equation}\label{metric:P_c,q}
	\|f-\bar f\|'''
   =\sup\set{\dfrac{\|f_n(u)-\bar f_n(u)\|}{\|u\|^{q+1}} \colon
      n \in \N, u \in X \setm{0}},
\end{equation}
for every $f=\prts{f_n}_{n \in \N}, \bar f=\prts{\bar f_n}_{n \in \N} \in
\cP_{c,q}$.

The purpose of this section is to see how the manifolds in
Theorem~\ref{thm:local} vary with the perturbations. To do this we consider two
sequence of perturbations $f, \bar f \in \cP_{c,q}$ and the functions $\phi$
and $\bar\phi$ given by Theorem~\ref{thm:local} when we perturb
equation~\eqref{eq:dyn} with $f$ and $\bar f$, respectively, and we compare the
distance between $\phi$ and $\bar\phi$ in the metric given
by~\eqref{def:metric:X_beta} with the distance between $f$ and $\bar f$ in the
metric given by~\eqref{metric:P_c,q}.

\begin{theorem}
    Let $c > 0$ and $q > 1$. Suppose that equation~\eqref{eq:lin:dif} admits a
    $(\mu,\nu)$-dichotomy for some $D \ge 1$, $a <  0\le b$ and $\eps > 0$ and
    that the hypothesis of Theorem~\ref{thm:local} are satisfied. Then,
    choosing  for $\delta>0$ sufficiently small, we have
		$$ \|\phi-\bar\phi\|' \le \|f-\bar f\|'''$$
   for every $f, \bar f \in \cP_{c,q}$, where $\phi, \bar \phi \in
   \cX_{\delta,\beta}$ are the functions given by Theorem~\ref{thm:local} for
   the same constant $C > D$ corresponding to the perturbations $f$ and $\bar
   f$, respectively.
\end{theorem}

\begin{proof}
   Let $n \in \N$ and $\xi \in B_n(\delta \beta_n)$. From~\eqref{eq:phi_n},
   putting for every $k \ge n$
      $$ \gamma_k
         := \|f_k(x^\phi_k(\xi), \phi_k(x^\phi_k(\xi)))
            - \bar f_k(x^{\bar\phi}_k(\xi),
            \bar\phi_k(x^{\bar\phi}_k(\xi)))\|,$$
   we obtain
	\begin{equation} \label{loc:eq:MajP3}
      \|\phi_n (\xi)-\bar \phi_n (\xi)\|
      \le \dsum_{k=n}^{+\infty} \|\cA_{k+1,n}^{-1} Q_{k+1}\| \,
         \gamma_k,
	\end{equation}
   where $x^\phi, x^{\bar\phi} \in \cB_{n,\delta,\beta}$ are the sequences of
   functions given by Lemma~\ref{lemma:Exist-Suc-x_m} associated with
   $(f,\phi)$ and $(\bar f, \bar \phi)$, respectively. By \eqref{metric:P_c,q},
   \eqref{cond-f-1}, \eqref{cond-phi-e-1}, \eqref{cond-phi-e-2},
   \eqref{cond-x_m-1a} and~\eqref{def:metric-of-B} we have for $k \ge n$
   \begin{equation} \label{norm_f_x,phi_-_bar-f_bar_x,bar-phi}
      \begin{split}
   		\gamma_k
   		& \le \|f_k(x^\phi_k(\xi),\phi_k(x^\phi_k(\xi)))
               - \bar f_k(x^\phi_k(\xi),\phi_k(x^\phi_k(\xi)))\|\\
            & \quad + \|\bar f_k(x^\phi_k(\xi),\phi_k(x^\phi_k(\xi)))
               - \bar f_k(x^{\bar\phi}_k
               (\xi),\bar\phi_k(x^{\bar\phi}_k(\xi)))\|\\
   		& \le 3^{q+1} \|f-\bar f\|'''  \|x^\phi_k (\xi)\|^{q+1}\\
            & \quad + 3^{q+1}c  \|x^\phi_k (\xi)-x^{\bar\phi}_k (\xi)\|
               (\|x^\phi_k (\xi)\|+\|x^{\bar\phi}_k (\xi)\|)^q \\
      		& \quad + 3^q c \|\phi-\bar \phi\|' \|x^{\bar\phi}_k (\xi)\|
               \prts{\|x^\phi_k (\xi)\|+\|x^{\bar \phi}_k (\xi)\|}^q \\
   		& \le 3^{q+1} C^{q+1} \delta^q \|f-\bar f\|''' \|\xi\|
               \pfrac{\mu_k}{\mu_{n-1}}^{aq+a}
               \nu_{n-1}^{\eps(q+1)} \beta_n^q\\
      		& \quad + 2^q 3^{q+1}c C^q \delta^q \|x^\phi-x^{\bar\phi}\|''
               \|\xi\| \pfrac{\mu_k}{\mu_{n-1}}^{aq+a} \nu_{n-1}^{\eps(q+1)}
               \beta_n^q\\
            & \quad + 2^q 3^q c C^{q+1} \delta^q \|\phi-\bar \phi\|' \|\xi\|
               \pfrac{\mu_k}{\mu_{n-1}}^{aq+a} \nu_{n-1}^{\eps (q + 1)}
               \beta_n^q
      \end{split}
   \end{equation}
   and using~\eqref{eq:b_k,mu_k,nu_k=1}, the last
   estimate,~\eqref{loc:eq:MajP3}, we get
   \begin{equation} \label{norm:phi(s,xi)-bar_phi(s,xi)}
      \begin{split}
         \|\phi_n (\xi)-\bar \phi_n (\xi)\|
   		& \le 3^{q+1} C^{q+1} D \delta^q \|f-\bar f\|''' \|\xi\|
            \mu_{n-1}^{-aq} \nu_{n-1}^{\eps (q + 1)} \beta_n^q
            \dsum_{k=n}^{+\infty} \mu_k^{aq} \nu_k^\eps  \\
   		& \quad + 2^q 3^{q+1}c C^{q} D \delta^q
            \|x^\phi-x^{\bar\phi}\|'' \|\xi\|
            \mu_{n-1}^{-aq} \nu_{n-1}^{\eps (q+1)} \beta_n^q
            \dsum_{k=n}^{+\infty} \mu_k^{aq} \nu_k^\eps  \\
         & \quad + 2^q 3^q c C^{q+1} D \delta^q \|\phi-\bar \phi\|' \|\xi\|
            \mu_{n-1}^{-aq} \nu_{n-1}^{\eps (q + 1)} \beta_n^q
            \dsum_{k=n}^{+\infty}
            \mu_k^{aq} \nu_k^\eps  \\
         & \le 3^{q+1} C^{q+1} D \delta^q \|f-\bar f\|''' \|\xi\|
            + 2^q 3^{q+1}c C^q D \delta^q \|x^\phi-x^{\bar\phi}\|'' \|\xi\|\\
         & \quad + 2^q 3^q c C^{q+1} D \delta^q \|\phi-\bar \phi\|' \|\xi\|.
      \end{split}
   \end{equation}
   Now, we will estimate $\|x^\phi-x^{\bar\phi}\|''$.
   By~\eqref{eq:dyn-split2a},
   ~\eqref{norm_f_x,phi_-_bar-f_bar_x,bar-phi} and~\eqref{eq:dich-1} we obtain
   for every $m \ge n$ and every  $\xi \in B_n(\delta \beta_n)$
	\begin{equation*}
		\begin{split}
         & \pfrac{\mu_m}{\mu_{n-1}}^{-a} \nu_{n-1}^{-\eps}
            \|x^\phi_m (\xi)-x^{\bar\phi}_m (\xi)\| \\
   		& \le \pfrac{\mu_m}{\mu_{n-1}}^{-a} \nu_{n-1}^{-\eps}
            \dsum_{k=n}^{m-1} \|\cA_{m,k+1}P_{k+1}\| \, \gamma_k\\
   		& \le 3^{q+1} C^{q+1} D \delta^q \|f-\bar f\|''' \|\xi\|
            \mu_{n-1}^{-aq} \nu_{n-1}^{\eps q} \beta_n^q
            \dsum_{k=n}^{m-1} \mu_k^{aq} \nu_k^\eps  \\
   		& \quad + 2^q  3^{q+1}c C^{q} D \delta^q
            \|x^\phi-x^{\bar\phi}\|'' \|\xi\|
            \mu_{n-1}^{-aq} \nu_{n-1}^{\eps q} \beta_n^q
            \dsum_{k=n}^{m-1} \mu_k^{aq} \nu_k^\eps  \\
         & \quad + 2^q 3^q c C^{q+1} D \delta^q \|\phi-\bar \phi\|' \|\xi\|
            \mu_{n-1}^{-aq} \nu_{n-1}^{\eps q} \beta_n^q \dsum_{k=n}^{m-1}
            \mu_k^{aq}
            \nu_k^\eps  \\
         & \le 3^{q+1} C^{q+1} D \delta^q \|f-\bar f\|''' \|\xi\|
            + 2^q 3^{q+1}c C^q D \delta^q \|x^\phi-x^{\bar\phi}\|'' \|\xi\|\\
         & \quad + 2^q 3^q c C^{q+1} D \delta^q \|\phi-\bar \phi\|' \|\xi\|
   	\end{split}
	\end{equation*}
   and this implies
	\begin{equation*}
		\begin{split}
         & \|x^\phi - x^{\bar\phi}\|''\\
         & \le 3^{q+1} C^{q+1} D \delta^q \|f-\bar f\|'''
            + 2^q 3^{q+1}c C^q D \delta^q \|x^\phi-x^{\bar\phi}\|''
            + 2^q 3^q c C^{q+1} D \delta^q \|\phi-\bar \phi\|'.
   	\end{split}
	\end{equation*}
   Thus, for $\delta>0$ such that $2^q 3^{q+1}c C^q D \delta^q < 1/2$ we have
      $$ \|x^\phi - x^{\bar\phi}\|''
         \le 2 \cdot 3^{q+1} C^{q+1} D \delta^q \|f-\bar f\|'''
            + 2^{q+1} 3^q c C^{q+1} D \delta^q \|\phi-\bar \phi\|'.$$
   It follows from the last estimate,~\eqref{norm:phi(s,xi)-bar_phi(s,xi)} and
   $2^q 3^{q+1}c C^q D \delta^q < 1/2$ that
      $$ \|\phi_n (\xi) - \bar\phi_n (\xi)\|
         \le 2 \cdot 3^{q+1} C^{q+1} D \delta^q \|f-\bar f\|''' \|\xi\|
            + 2^{q+1} 3^q c C^{q+1} D \delta^q \|\phi-\bar\phi\|' \|\xi\|$$
   for every $n \in \N$ and every $\xi \in B_n (\delta \beta_n)$. Hence we get
      $$ \|\phi - \bar\phi\|'
         \le 2 \cdot 3^{q+1} C^{q+1} D \delta^q \|f-\bar f\|'''
            + 2^{q+1} 3^q c C^{q+1} D \delta^q \|\phi-\bar\phi\|'$$
   and choosing $\delta>0$ such that $2^{q+1} 3^q c
   C^{q+1} D \delta^q < 1/2$  we obtain
   	$$	\|\phi-\bar\phi\|'
         \le 4 \cdot 3^{q+1} C^{q+1} D \delta^q \|f-\bar f\|'''.$$
   To finish the proof we have to choose $\delta > 0$ such that $4 \cdot
   3^{q+1} C^{q+1} D \delta^q \le 1$.
\end{proof}

\section*{Acknowledgments}
This work was supported by Centro de Matem\'atica da Universidade da Beira
Interior.
\bibliographystyle{elsart-num-sort}

\end{document}